\documentclass[12pt, reqno]{amsart}
\usepackage{amsmath, amsthm, amscd, amsfonts, amssymb, graphicx, color}
\usepackage[bookmarksnumbered,plainpages]{hyperref}
\usepackage[all]{xy}

\newtheorem{theorem}{Theorem}[section]
\newtheorem{lemma}[theorem]{Lemma}
\newtheorem{proposition}[theorem]{Proposition}
\newtheorem{corollary}[theorem]{Corollary}
\usepackage{enumerate}
\theoremstyle{definition}
\newtheorem{definition}[theorem]{Definition}
\newtheorem{example}[theorem]{Example}

\newtheorem{Theorem}{\quad Theorem}[section]

\newtheorem{remark}[Theorem]{\quad Remark}
\numberwithin{equation}{section}
\input{mathrsfs.sty}
\begin{document}
\title[More on $\omega$-orthogonality and $\omega$-parallelism]
{More on $\omega$-orthogonalities and $\omega$-parallelism}

\author[M. Torabian, M. Amyari and M. Moradian Khibary]
{M. Torabian$^1$, M. Amyari $^{2*}$ and M. Moradian Khibary$^3$}

\address{$^1,^{2*}$ Department of Mathematics, Mashhad Branch,
Islamic Azad University, Mashhad, Iran}
\email{m.torabian93@gmail.com}
\email{amyari@mshdiau.ac.ir and maryam\_amyari@yahoo.com}

\address{$^3$ Department of Mathematics, Farhangian University, Mashhad, Iran}
\email{mmkh926@gmail.com}


\keywords{Hilbert space, numerical radius Birkhoff orthogonality, numerical radius parallelism.\\
$*$Corresponding author}

\begin{abstract}
We investigate some aspects of various numerical radius orthogonalities and numerical radius parallelism
 for bounded linear operators on a Hilbert space $\mathscr{H}$. Among several results, we show that if
 $T,S\in \mathbb{B}(\mathscr{H})$ 
 and $M^*_{\omega(T)}=M^*_{\omega(S)}$, then $T\perp_{\omega B} S$ if and only if $S\perp_{\omega B} T$, where
$M^*_{\omega(T)}=\{\{x_n\}:\,\,\,\|x_n\|=1, \lim_n|\langle Tx_n, x_n\rangle|=\omega(T)\}$, and $\omega(T)$ is the numerical radius of $T$ and $\perp_{\omega B}$
is the numerical radius Birkhoff orthogonality.
\end{abstract}

\maketitle
\section{Introduction and preliminary}
Let $(\mathscr{H}, \langle  \cdot, \cdot \rangle)$ be a complex Hilbert space and $\mathbb{B}(\mathscr{H})$ be the $C^*$-algebra of all bounded linear operators.
The numerical range of an operator $T\in \mathbb{B}(\mathscr{H})$ is the subset of the 
complex numbers $\mathbb{C}$ given by $W(T)=\{\langle Tx, x\rangle:\,\,x\in \mathscr{H},\,\, \|x\|=1\}$ and the numerical radius of $T$ is defined by $$\omega(T)=\sup\{|\langle Tx, x\rangle|:\,\,\|x\|=1\}.$$
It is known that $\omega(T)$ is a norm on $\mathbb{B}(\mathscr{H})$ satisfying
\begin{equation}\label{11}
\omega(T)\leq \|T\| \leq2\omega(T),
\end{equation}
where $\|T\|$ denotes the operator norm of $T$. If $T$ is self-adjoint, then $\omega(T)=\|T\|$  \cite{GU}.
The authors in \cite[Theorem 1.1]{JS} proved that if
$
T=\begin{bmatrix}
a&b\\
0&d
\end{bmatrix}
$
for all $a, b, d\in\mathbb{C}$, then
\begin{align}\label{12}
\omega(T)=\dfrac{1}{2}|a+d|+\dfrac{1}{2}\sqrt{|a-d|^2+|b|^2}.
\end{align}
Let $T,\,\,S\in\mathbb{B}(\mathscr{H})$, if $S^*T=0$, then we say that $T$ is {\it orthogonal} to $S$ and we write $S\perp T$.
We say that $T$ is {\it Birkhoff orthogonal} to $S$, if $\|T+\lambda S\|\geq \|T\|$ for all $\lambda\in\mathbb{C}$ and we write $T\perp_B S$.
 It is easy to show that the Birkhoff orthogonality is homogeneous (that is, if $T\perp_B S$,
 then $\alpha T\perp_B \beta S$ for all $\alpha,\beta\in\mathbb{C}$). In general, the Birkhoff orthogonality is neither additive nor symmetric.
 Turn\^sek  \cite{AT} proved the Birkhoff orthogonality is symmetric
 if and only if one of operators is a scalar multiple of an isometry or coisometry.

Let $(X, \|.\|)$ be a normed space. An element $x\in X$ is said to be {\it norm-parallel} to an element $y\in X$, denote by $x\parallel y$, if
$$\|x+\lambda y\|=\|x\|+\|y\|\quad\mbox{for some}\quad\lambda\in\mathbb{T},$$
where $\mathbb{T}=\{\mu\in\mathbb{C}:|\mu|=1\}.$\\
For two operators $T,\, S\in \mathbb{B}(\mathscr{H})$,
the authors in \cite{MM},  introduced the notion of parallelism.
 The operator $T$ is called the {\it numerical radius parallel} to $S$, denote by $T\parallel_\omega S$, if
$$\omega(T+\lambda S)=\omega(T)+\omega(S)\quad\mbox{for some}\quad\lambda\in\mathbb{T}$$
Now, we give some types of orthogonalities for operators on a Hilbert space $\mathscr{H}$ based on the notion of the numerical radius.

\begin{definition}
Suppose that $T, S\in \mathbb{B}(\mathscr{H})$.\\

(i) The operator $T$ is called the {\it numerical radius Pythagorean orthogonal} to $S$, denoted by $T\perp_{\omega p}S$, if $$\omega^2(T+S)=\omega^2(T)+\omega^2(S).$$
This notion is introduced in \cite{AM}.\\
(ii) The operator $T$ is called the {\it numerical radius Birkhoff orthogonality} to $S$, denoted by $T\perp_{\omega B}S$, if
$$\omega(T+\lambda S)\geq\omega(T)\quad\mbox{for all}\quad\lambda\in\mathbb{C}.$$
\end{definition}

 Recently, some mathematicians studied numerical radius orthogonality and  numerical radius parallelism for
 operators, for instance, see \cite{BS, BC, KM,  KS, AZ}.

\section{Numerical radius Birkhoff orthogonality}
In this section, we aim to prove some properties of numerical radius Birkhoff orthogonality $\perp_{B\omega}$. Suppose that $T,\,S\in \mathbb{B}(\mathscr{H})$.
 It is easy to see that $\perp_{\omega B}$ is nondegenerate ($T\perp_{\omega B}T$ if and only if $T=0$) and homogenous ($T\perp_{\omega B}S  \Rightarrow \alpha T\perp_{\omega B} \beta S$ for all $\alpha,\beta\in\mathbb{C}$). The authors of \cite{AM} showed that (i) $T\perp_{\omega B}S$ if and only if $T^*\perp_{\omega B}S^*$; (ii) if $T$ is self-adjoint, then $T\perp_{\omega B}S$ implies that $T\perp_B S$; (iii) if $T^2=0$, then
$T\perp_B S$ implies that $T\perp_{\omega B}S$.\\
In addition, if $S\perp T$ and $S$ is surjective, then $S^*T=0$ and
\begin{align*}
\omega(T)&=\sup\{|\langle Tx,x\rangle|:\,\, \|x\|=1\}=\sup\{|\langle TSx_1,Sx_1\rangle|:\,\, \|Sx_1\|=1\}\\
&=\sup\{|\langle S^*TSx_1,x_1\rangle|:\,\, \|Sx_1\|=1\}=0.
\end{align*}
Hence $T=0$. As a consequence, we get $S\perp_{\omega B} T$.\\
The following example shows that $\perp$ does not imply $\perp_{\omega B}$, in general. Indeed, the condition of surjectivity is necessary.
\begin{example}
Suppose that
$
S=\begin{bmatrix}
0&-1\\
0&1
\end{bmatrix}\quad
T=\begin{bmatrix}
0&1\\
0&1
\end{bmatrix}
$
 are in $M_2(\mathbb{C})$ and $S$ is not surjective. Then
$S^*T=0,$
that is $S\perp T$. By equation \eqref{12}, $\omega(S)=\frac{1+\sqrt{2}}{2}$
and for $\lambda=-1$, we get $\omega(S-T)=1$.
Therefore $\omega(S)>\omega(S-T)$. Hence, $S\not\perp_{\omega B}T$.
\end{example}
The following theorem gives a characterization of the numerical radius Birkhoff orthogonality of
operators on a  Hilbert space.
\begin{theorem}\cite[Theorem 2.3]{AM}\label{22}
Let $T, S\in \mathbb{B}(\mathscr{H})$. Then  $T\perp_{\omega B}S$
 if and only if for each $\theta \in [0, 2\pi)$, there exists a sequence $\{x^{\theta}_n\}_{n\in\mathbb{N}}$ of unit vectors in $\mathscr{H}$ such that the
following two conditions hold:

(i) $\displaystyle\lim_n |\langle Tx^{\theta}_n, x^{\theta}_n\rangle|=\omega(T)$,

(ii) $\displaystyle\lim_n {\rm  Re}\{e^{-i\theta} \langle Tx^{\theta}_n, x^{\theta}_n\rangle\overline{\langle Sx^{\theta}_n, x^{\theta}_n\rangle}\}\geq 0$.
\end{theorem}
Now, we show that  $`` \perp_{\omega B}"$ is not symmetric, in general.
\begin{example}
Suppose that
$ T=
\begin{bmatrix}
1&0\\
0&0
\end{bmatrix}
\,\, and \,\,
S=
\begin{bmatrix}
0&1\\
0&-1
\end{bmatrix}
$
are in $M_2(\mathbb{C})$. Then equation \eqref{12} implies that
$\omega(T)=1$, and
$\omega(T+\lambda S)=\dfrac{1}{2}|1-\lambda|+\dfrac{1}{2}\sqrt{|1+\lambda|^2+|\lambda|^2}\geq\dfrac{1}{2}|1-\lambda+1+\lambda|=1$
for each $\lambda \in \mathbb{C}$.
Hence $\omega(T+\lambda S)\geq \omega(T)$ and so $T\perp_{\omega B} S$.
On the other hand,
$\omega(S)=\dfrac{1+\sqrt{2}}{2}=1.207.$
For $\lambda=1$, we get $\omega(S+ T)=\dfrac{\sqrt{5}}{2}=1.118<1.207$.
Hence, $\omega(S+\lambda T)< \omega(S)$, and so $S\not\perp_{\omega B} T$.
\end{example}
We can show that $\perp_{\omega B}$ is symmetric if one of the operators is identity.
\begin{proposition}
Let $T\in \mathbb{B}(\mathscr{H})$. If $T\perp_{\omega B}I$, then $I\perp_{\omega B}T$.
 \begin{proof}
Let $\theta \in [0, 2\pi)$ be given. Since $T\perp_{\omega B}I$, for $\varphi=2\pi -\theta $  there exists a sequence $\{x^{\varphi}_n\}_{n\in\mathbb{N}}$ of unit vectors in $\mathscr{H}$ such that
$\displaystyle\lim_n |\langle Tx^{\varphi}_n, x^{\varphi}_n\rangle|=\omega(T)$ and
 $\displaystyle\lim_n{\rm  Re}\{e^{-i\varphi}\langle Tx^{\varphi}_n, x^{\varphi}_n\rangle  \overline{\langle Ix^{\varphi}_n, x^{\varphi}_n\rangle}\}\geq 0$.
 Put $x^{\theta}_n=x^{\varphi}_n$. Then $\omega(I)=1=\displaystyle\lim_n |\langle Ix^{\theta}_n, x^{\theta}_n\rangle|$ and
 $\displaystyle\lim_n{\rm  Re}\{e^{-i\theta}\langle Ix^{\theta}_n, x^{\theta}_n\rangle  \overline{\langle Tx^{\theta}_n, x^{\theta}_n\rangle}\}=
 \displaystyle\lim_n{\rm  Re}\{\overline{e^{-i\theta}\langle Ix^{\theta}_n, x^{\theta}_n\rangle}\langle Tx^{\theta}_n, x^{\theta}_n\rangle\}\geq 0$. That is $I\perp_{\omega B}T$.
\end{proof}
\end{proposition}
Let  $T\in \mathbb{B}(\mathscr{H})$.  The set of all sequences in the closed unit ball $\mathscr{H}$ at which $T$ attains numerical radius is denoted by
$$M^*_{\omega(T)}=\{\{x_n\}:\,\,\,\|x_n\|=1, \lim_n|\langle Tx_n, x_n\rangle|=\omega(T)\}.$$

Suppose that $T\perp_{\omega B} S$. The following theorem gives some conditions under which $S\perp_{\omega B} T$.
\begin{lemma}
Let $T,\, S \in \mathbb{B}(\mathscr{H})$ and $T\perp_{\omega B} S$. If
for each sequence $\{x_n\}\in M^*_{\omega(T)}$ there exists a sequence  $\{y_n\}\in M^*_{\omega(S)}$
with $\lim_n|\langle x_n, y_n\rangle|=1$, then $S\perp_{\omega B} T$.
\begin{proof}
For each $\lambda \in \mathbb{C}$ there exists $\theta \in [0, 2\pi)$ such that $\lambda=e^{-i\theta}  |\lambda|$.
Since $T\perp_{\omega B} S$, there exists $\{x^{\theta}_n\}\in M^*_{\omega(T)}$ such that
 $$\lim_n|\langle Tx^{\theta}_n, x^{\theta}_n\rangle|=\omega(T) \,\,\text{and}\,\,\lim_n {\rm  Re} \{e^{-i\theta}\langle Tx^{\theta}_n, x^{\theta}_n\rangle\overline{\langle S x^{\theta}_n, x^{\theta}_n\rangle}\}\geq 0.$$
Hence
\begin{align*}
\omega^2(S+\lambda T) &\geq  \lim_n|\langle (S+\lambda T) x^{\theta}_n, x^{\theta}_n\rangle|^2\\
&=\lim_n\Big(|\langle S x^{\theta}_n, x^{\theta}_n\rangle|^2+2|\lambda|{\rm  Re} \{e^{-i\theta}\langle Tx^{\theta}_n, x^{\theta}_n\rangle\overline{\langle S x^{\theta}_n, x^{\theta}_n\rangle}\}+|\lambda|^2|\langle Tx^{\theta}_n, x^{\theta}_n\rangle|^2\Big)\\
& \geq \lim_n|\langle S x^{\theta}_n, x^{\theta}_n\rangle|^2.
\end{align*}

By the assumption, there exists $\{y_n\}$ with $\|y_n\|=1$ such that $\lim_n|\langle Sy_n, y_n\rangle|=\omega(S)$
 and $\lim_n|\langle x^{\theta}_n, y_n\rangle|=1$.
For each $n$, we can write $y_n=\alpha_nx^{\theta}_n+z_n$ with $\langle x^{\theta}_n, z_n\rangle=0$ for some $\alpha_n \in \mathbb{C}$.
Then $1=\lim_n|\langle x^{\theta}_n, y_n\rangle|=\lim_n|\alpha_n|$. From $1=\|y_n\|^2=|\alpha_n|^2+\|z_n\|^2$ we conclued that
$\lim_n\|z_n\|=0$. The Cauchy–Schwarz inequality implies that
\begin{align*}
\omega(S)=\lim_n|\langle Sy_n, y_n\rangle|&=\lim_n|\langle S(\alpha_nx^{\theta}_n+z_n), \alpha_nx^{\theta}_n+z_n\rangle|\\
&=\lim_n||\alpha_n|^2\langle Sx^{\theta}_n, x^{\theta}_n\rangle+\alpha_n\langle Sx^{\theta}_n, z_n\rangle+\bar{\alpha_n}\langle Sz_n, x^{\theta}_n\rangle+\langle Sz_n,z_n\rangle|\\
& =\lim_n|\langle Sx^{\theta}_n, x^{\theta}_n\rangle|.
\end{align*}
Therefore $\omega(S) \leq \omega(S+\lambda T)$, that is $S\perp_{\omega B} T$.
\end{proof}
\end{lemma}
As a consequence of Theorem \ref{22}, we get the following.
\begin{theorem}
Let $T,\, S \in \mathbb{B}(\mathscr{H})$ and $M^*_{\omega(T)}=M^*_{\omega(S)}$.
Then $T\perp_{\omega B} S$ if and only if $S\perp_{\omega B} T$.
 \end{theorem}
 Now, we recall the following proposition that we will need in what follows.
\begin{proposition}\label{23} \cite[Proposition 3.6]{AA}
Let $T, S\in \mathbb{B}(\mathscr{H})$. Then the equality $\omega(T+S)=\omega(T)+\omega(S)$ holds if and
only if there exists a sequence $\{x_n\}_{n\in\mathbb{N}}$ of unit vectors in $\mathscr{H}$ such that
$$\displaystyle\lim_n \langle Tx_n, x_n\rangle\overline{\langle Sx_n, x_n\rangle}=\omega(T)\omega(S)$$
\end{proposition}

\begin{theorem}
Let $T,S\in \mathbb{B}(\mathscr{H})$. Then the following statements hold:

(i) If $T\perp_{\omega B}(\omega(T)S-\omega(S)T)$, then $\omega(T+S)=\omega(T)+\omega(S)$.

(ii) If $\omega(e^{-i\theta}T+S)=\omega(T)+\omega(S)$ for each $\theta \in [0, 2\pi)$, then $T\perp_{\omega B}(\omega(S)T-\omega(T)S)$.
\end{theorem}
\begin{proof}
 (i): Let $T\perp_{\omega B}(\omega(T)S-\omega(S)T)$.
Then for $\theta=0$, there exists a sequence $\{x_n\}_{n\in\mathbb{N}}$ of unit vectors in $\mathscr{H}$ such that
$$\lim_n |\langle Tx_n, x_n\rangle|=\omega(T)\,\, \text{and}\,\,\displaystyle\lim_n{\rm  Re}\left\{ {\langle Tx_n, x_n\rangle}\overline{\langle(\omega(T)S-\omega(S)T)x_n, x_n\rangle}\right\}\geq 0.$$
Therefore,
\small{
\begin{align*}
0 \leq \lim_n{\rm  Re}\left\{ \langle Tx_n, x_n\rangle \overline{\langle(\omega(T)S-\omega(S)T)x_n, x_n\rangle}\right\}&=\lim_n{\rm  Re}\left\{\omega(T) \langle Tx_n, x_n\rangle  \overline{\langle Sx_n, x_n\rangle}\right\}\\
&\quad -\lim_n{\rm  Re}\left\{\omega(S)\langle Tx_n, x_n\rangle\overline{\langle Tx_n, x_n\rangle}\right\}\\
&=\omega(T)\lim_n{\rm  Re}\left\{\langle Tx_n, x_n\rangle  \overline{\langle Sx_n, x_n\rangle}\right\}-\omega(S)\omega^2(T),
\end{align*}
}
from which we get $\lim_n{\rm  Re}\left\{\langle Tx_n, x_n\rangle  \overline{\langle Sx_n, x_n\rangle}\right\}\geq \omega(S)\omega(T)$. Hence, by passing to subsequences if necessary, we obtain
$$\omega(S)\omega(T)\leq \lim_n{\rm  Re}\left\{\langle Tx_n, x_n\rangle  \overline{\langle Sx_n, x_n\rangle}\right\}\leq \lim_n|\langle Tx_n, x_n\rangle  \overline{\langle Sx_n, x_n\rangle}|
\leq \omega(S)\omega(T).$$
Thus
$$\lim_n \langle Tx_n, x_n\rangle\overline{\langle Sx_n, x_n\rangle}=\omega(T)\omega(S).$$
It follows from Proposition \ref{23} that $\omega(T+S)=\omega(T)+\omega(S)$.

(ii): Let $\theta \in [0, 2\pi)$ such that $\omega(e^{-i\theta}T+S)=\omega(T)+\omega(S)$. Then there exists a sequence $\{x^{\theta}_n\}_{n\in\mathbb{N}}$ of unit vectors in $\mathscr{H}$
such that
\begin{align*}
\omega(T)+\omega(S)=\omega(e^{-i\theta}T+S)&=\displaystyle\lim_n|\langle (e^{-i\theta} T+S )x^{\theta}_n, x^{\theta}_n\rangle|\\
&\leq \displaystyle\lim_n |\langle e^{-i\theta}Tx^{\theta}_n, x^{\theta}_n\rangle|+\displaystyle\lim_n |\langle Sx^{\theta}_n, x^{\theta}_n\rangle|\\
&\leq \omega(T)+\omega(S),
\end{align*}
where we use Bolzano--Weierstrass theorem and pass to subsequences of $\{x^{\theta}_n\}_{n\in\mathbb{N}}$ if necessary.
 Hence $\displaystyle\lim_n |\langle e^{-i\theta}Tx^{\theta}_n, x^{\theta}_n\rangle|=\omega(T)$ and $\displaystyle\lim_n |\langle Sx^{\theta}_n, x^{\theta}_n\rangle|=\omega(S)$.
On the other hand,
\small{
\begin{align*}
(\omega(T)+\omega(S))^2=\omega^2(e^{-i\theta}T+S)&= \lim_n| \langle (e^{-i\theta}T+S)x^{\theta}_n, x^{\theta}_n\rangle|^2\\
&=\lim_n\left(|\langle e^{-i\theta} Tx^{\theta}_n, x^{\theta}_n\rangle|^2+2{\rm  Re}\left\{e^{-i\theta}\langle Tx^{\theta}_n, x^{\theta}_n\rangle \overline{\langle Sx^{\theta}_n, x^{\theta}_n\rangle}\right\}+|\langle Sx^{\theta}_n, x^{\theta}_n\rangle|^2\right)\\
&\leq \lim_n\left(|e^{-i\theta}\langle Tx^{\theta}_n, x^{\theta}_n\rangle|^2+2|e^{-i\theta}\langle Tx^{\theta}_n, x^{\theta}_n\rangle\overline{\langle Sx^{\theta}_n, x^{\theta}_n\rangle}|+|\langle Sx^{\theta}_n, x^{\theta}_n\rangle|^2\right)\\
&\leq \omega^2(T)+2\omega(T)\omega(S)+\omega^2(S)=(\omega(T)+\omega(S))^2.
\end{align*}
}
Hence, $\lim_n{\rm  Re}\left\{e^{-i\theta} {\langle Tx_n, x_n\rangle}\overline{\langle Sx_n, x_n\rangle}\right\}=\omega(T)\omega(S)$.
In addition,
\begin{align*}
\lim_n{\rm  Re}\left\{e^{-i\theta} {\langle Tx^{\theta}_n, x^{\theta}_n\rangle}\overline{\langle(\omega(T)S-\omega(S)T)x^{\theta}_n, x^{\theta}_n\rangle}\right\}&=
\lim_n{\rm  Re}\Big\{e^{-i\theta}\Big(\omega(T)\langle Tx^{\theta}_n, x^{\theta}_n\rangle  \overline{\langle Sx^{\theta}_n, x^{\theta}_n\rangle}\\
&\quad -\omega(S)\langle Tx^{\theta}_n, x^{\theta}_n\rangle\overline{\langle Tx^{\theta}_n, x^{\theta}_n\rangle}\Big)\Big\}\\
&=\omega(T)\Big(\lim_n{\rm  Re}\{e^{-i\theta}\langle Tx^{\theta}_n, x^{\theta}_n\rangle\overline{\langle Sx^{\theta}_n, x^{\theta}_n\rangle}\}\Big)\\
&\quad -\omega(S)\Big( \lim_n{\rm  Re}\{e^{-i\theta}|\langle Tx^{\theta}_n, x^{\theta}_n\rangle|^2\}\Big)\\
&= \omega^2(T)\omega(S)\Big(1-{\rm  Re}(e^{-i\theta})\Big) \geq0
\end{align*}
Note that $|{\rm  Re}(e^{-i\theta})|\leq 1$. Theorem \ref{22} implies that $T\perp_{\omega B}(\omega(T)S-\omega(S)T)$.
\end{proof}

\begin{proposition}\label{24}
Let $T, S\in \mathbb{B}(\mathscr{H})$ and $T$ be a positive operator. Then $T\perp_{\omega B} S$ if and only if $T+I\perp_{\omega B} S$.
\end{proposition}
\begin{proof}
 Let $T$ be positive and $T\perp_{\omega B}S$. Then for each $\theta\in[0,2\pi)$, there exists $\{x_n^{\theta}\}$ with $\|x_n^{\theta}\|=1$ such that
 \begin{align}\label{25}
\lim_n|\langle Tx_n^{\theta},x_n^{\theta}\rangle|=\omega(T)\,\,\text{and}\,\,
\lim_n {\rm  Re} \{e^{-i\theta} \langle Tx_n^{\theta},x_n^{\theta}\rangle\overline{\langle S x_n^{\theta},x_n^{\theta}\rangle}\}\geq 0.
\end{align}
By passing to subsequences we may assum that $\displaystyle\lim_n\langle S x_n^{\theta},x_n^{\theta}\rangle$ exists. Since $T$ is positive, $\omega(T+I)=\omega(T)+1$. Hence
\begin{align*}
\lim_n|\langle (T+I) x_n^{\theta}, x_n^{\theta}\rangle|&=\lim_n|\langle T x_n^{\theta},x_n^{\theta}\rangle+\langle I x_n^{\theta},x_n^{\theta}\rangle|\\
&=\lim_n|\langle T x_n^{\theta}, x_n^{\theta}\rangle+1|\\
&=\omega(T)+1=\omega(T+I).
\end{align*}
Furthermore,
$$\lim_n {\rm  Re} \{e^{-i\theta}\langle (T+I) x_n^{\theta},x_n^{\theta}\rangle\overline{\langle S x_n^{\theta},x_n^{\theta}\rangle}\}=\lim_n\Big({\rm  Re}\{e^{-i\theta}\langle T x_n^{\theta},x_n^{\theta}\rangle\overline{\langle S x_n^{\theta},x_n^{\theta}\rangle}\}
+ {\rm  Re}\{ e^{-i\theta}\overline{\langle S x_n^{\theta},x_n^{\theta}\rangle}\}\Big)\geq 0.$$
 Note that since $\langle T x_n^{\theta},x_n^{\theta}\rangle\geq 0$, inequality  (\ref{25})
  implies that $\lim_n {\rm  Re} \{e^{-i\theta}\overline{\langle S x_n^{\theta},x_n^{\theta}\rangle\}}\geq 0$.\\
Thus $T+I\perp_{\omega B} S$.\\
Conversely, let $T+I\perp_{\omega B} S$. Let $\lambda\in\mathbb{C}$. We have $\omega(T+I+\lambda S)>\omega(T+I)$.
In addition,
$$\omega(T)+1=\omega(T+I)\leq\omega(T+I+\lambda S)\leq\omega(T+\lambda S)+1.$$
Hence $\omega(T)\leq\omega(T+\lambda S)$. Thus $ T\perp_{\omega B} S $.
\end{proof}
In the following example, we show that the positivity condition on $T$ cannot be omitted in Proposition \ref{24}.
\begin{example}
Let $T=\begin{bmatrix}
1&0\\
0&-1
\end{bmatrix}$, and
$ S=\begin{bmatrix}
1&1\\
0&1
\end{bmatrix}$.

Then $\omega(T)=1$ and $\omega(T+\lambda S)=|\lambda|+\dfrac{1}{2}\sqrt{4+|\lambda|^2}$ for each $\lambda\in\mathbb{C}$. Hence $ \omega(T+\lambda S)\geq 1$ and so $T\perp_{\omega B} S$. Furthermore,
 $T+I=\begin{bmatrix}
2&0\\
0&0
\end{bmatrix}$
and $\omega(T+I)=2$. On the other hand,
$$\omega(T+I+\lambda S)=\omega\left( \begin{bmatrix}
2+\lambda&\lambda\\
0&\lambda
\end{bmatrix}\right)=\dfrac{1}{2}|2+2\lambda|+\dfrac{1}{2}\sqrt{4+|\lambda|^2}.$$
If $\lambda=-1$, then $\omega(T+I-S)=\sqrt{\dfrac{5}{2}}<2$. So $T+I\not\perp_{\omega B} S$.
\end{example}

\begin{remark}
At the end of this section, we show that for some special classes of operators $T\not\perp_{\omega B} S$.
Let $T, S\in \mathbb{B}(\mathscr{H})$ be positive operators and $\omega(S)\neq0$. Then $T\not\perp_{\omega B} S$.
To see this on the contrary, without loss of generality, assume that $T\perp_{\omega B} S$.
Then there exists a sequence
$\{x_n\}$ in $\mathscr{H}$ such that for each $\lambda\in\mathbb{C}$,
$$\lim_n|\langle Tx_n, x_n\rangle+\lambda\langle Sx_n, x_n\rangle|=\omega(T+\lambda S).$$
put $\lambda=-\dfrac{\omega(T)}{\omega(S)},\,\,\alpha=\dfrac{\lim_n \langle Sx_n, x_n\rangle}{\omega(S)},\,\,\beta=\dfrac{\lim_n\langle Tx_n, x_n\rangle}{\omega(T)}, \,\, (\alpha,\beta\in[0,1])$.
Then
\begin{align*}
\omega(T+\lambda S)&=\lim_n\Big|\langle Tx_n, x_n\rangle-\dfrac{\omega(T)}{\omega(S)}\langle Sx_n, x_n\rangle\Big|\\
&=|\beta-\alpha|\omega(T)<\omega(T),
\end{align*}
whis gives a contradiction.\\
Moreover, it is easy to show that in special case $\perp_{\omega B}$ is additive. Suppose that   $T, S,\,U\in \mathbb{B}(\mathscr{H})$ such that $S, U$ are two positive operators. Then both $T\perp_{\omega B} S$ and $T\perp_{\omega B} U$ if and only if  $T\perp_{\omega B} S+U$.
\end{remark}
\section{Numerical radius parallelism and Numerical radius orthogonalities}

In this section, we state some relations between numerical radius parallelism and some types of
numerical radius orthogonalities.

\begin{theorem}\label{31} {\rm \cite[Theorem 2.2]{MM}}
Let $T,S\in \mathbb{B}(\mathscr{H})$. Then the following statements are equivalent:

(i)  $T \parallel_{\omega} S$.

(ii) There exists a sequence of unit vectors $\{x_n\}$ in $\mathscr{H}$ such that
\begin{align*}
\lim_{n\rightarrow\infty} \big|\langle Tx_n, x_n\rangle\langle Sx_n, x_n\rangle\big| = \omega(T)\omega(S).
\end{align*}
In addition, if $\{x_n\}$ is a sequence of unit vectors in $\mathscr{H}$ satisfying $(ii)$,
then it also satisfies $\displaystyle{\lim_{n\rightarrow\infty}}|\langle Tx_n, x_n\rangle| = \omega(T)$
and $\displaystyle{\lim_{n\rightarrow\infty}}|\langle Sx_n, x_n\rangle| = \omega(S)$.
\end{theorem}
\begin{theorem}
Let $T,S\in \mathbb{B}(\mathscr{H})$ such that $\lim_n ({\rm  Re} \langle Tx_n, x_n\rangle\overline{\langle Sx_n, x_n\rangle})=0$ for each sequence $\{x_n\}\in\Big(M^*_{\omega(T)}\cap M^*_{\omega(S)}\Big)\cup M^*_{\omega(S+T)}$. Then the following statements are equivalent.

(i) $T\parallel_{\omega} S$;

(ii) $T\perp_{\omega p} S$.

\end{theorem}
\begin{proof}
(i) $\Rightarrow$ (ii): Let $T\parallel_{\omega} S$. Employing Theorem \ref{31}, we get a sequence $\{x_n\}\in\Big(M^*_{\omega(T)}\cap M^*_{\omega(S)}\Big)$ such that
 $\lim_n|\langle Tx_n, x_n\rangle|=\omega(T)$ and $ \lim_n|\langle Sx_n, x_n\rangle|=\omega(S)$. Hence
\begin{align*}
\lim_n|\langle (T+S) x_n, x_n\rangle|^2&=\lim_n\Big(|\langle Tx_n, x_n\rangle|^2+2{\rm  Re} \langle Tx_n, x_n\rangle\overline{\langle Sx_n, x_n\rangle}+|\langle Sx_n, x_n\rangle|^2\Big)\\
&=\omega^2(T)+\omega^2(S).
\end{align*}
Let $\{y_n\}\in M^*_{\omega(T+S)}$. Hence
\begin{align*}
\omega^2(T)+\omega^2(S)&=\lim_n|\langle (T+S) x_n, x_n\rangle|^2 \leq\omega^2(T+S)\\
&=\lim_n|\langle (T+S) y_n,y_n\rangle|^2=\lim_n|\langle Ty_n, y_n\rangle|^2+\lim_n|\langle Sy_n, y_n\rangle|^2\\
&\leq\omega^2(T)+\omega^2(S).
\end{align*}
Then $\omega^2(T+S)=\omega^2(T)+\omega^2(S)$, which means that $T\perp_{\omega p}S$.

(ii) $\Rightarrow$ (i): Let $T\perp_{\omega p} S$ and $\{x_n\}\in M^*_{\omega(T+S)}$. Then
\begin{align*}
\omega^2(T)+\omega^2(S)=\omega^2(T+S)=\lim_n|\langle (T+S) x_n, x_n\rangle|^2&=\lim_n|\langle Tx_n, x_n\rangle|^2+|\langle Sx_n, x_n\rangle|^2\\ &\leq\omega^2(T)+\omega^2(S).
\end{align*}
Therefore $\lim_n|\langle Tx_n, x_n\rangle|=\omega(T)$ and $\lim_n|\langle Sx_n, x_n\rangle|=\omega(S)$. Theorem \ref{31} implies that $T\parallel_{\omega} S$.
\end{proof}
\begin{corollary}
Let $T,\, S\in \mathbb{B}(\mathscr{H})$. Let $T$ be Hermitian and $S$ be skew-Hermitian. Then $T\parallel_{\omega} S$  if and only if $T\perp_{\omega p} S$.
\begin{proof}
Since $T=T^*$ and $S=-S^*$ we have $\langle Tx, x\rangle$ is a real number and
$\langle Sx, x\rangle$ is a purely imaginary number for every $x\in \mathbb{B}(\mathscr{H})$.
\end{proof}
\end{corollary}
\begin{theorem}
Let $T,\, S \in \mathbb{B}(\mathscr{H})$. Then the following statements are equivalent:

(i) $\omega(T+S)=\omega(T-S)$ and $M^*_{\omega(T-S)}=M^*_{\omega (T+S)}$;

(ii)  $\lim_n {\rm  Re} \langle Tx_n, x_n\rangle\overline{\langle Sx_n, x_n\rangle}=0$ for every sequence $\{x_n\}\in M^*_{\omega(T-S)}\cup M^*_{\omega (T+S)}$,

\end{theorem}
\begin{proof}
(i)$\Rightarrow$ (ii): Let $\{x_n\}\in M^*_{\omega(T-S)}\cup M^*_{\omega (T+S)}$. Since $M^*_{\omega(T-S)}=M^*_{\omega (T+S)}$, we get $\{x_n\}\in M^*_{\omega(T-S)}\cap M^*_{\omega (T+S)}$ and
\begin{align*}
\lim_n\Big(|\langle Tx_n, x_n\rangle|^2-2{\rm  Re} \langle Tx_n, x_n\rangle\overline{\langle Sx_n, x_n\rangle}+|\langle Sx_n, x_n\rangle|^2=\omega^2(T-S)\\
=\omega^2(T+S)=\lim_n \Big(|\langle Tx_n, x_n\rangle|^2+2{\rm  Re} \langle Tx_n, x_n\rangle\overline{\langle Sx_n, x_n\rangle}+|\langle Sx_n, x_n\rangle|^2\Big).
\end{align*}
Therefore $\lim_n {\rm  Re} \langle Tx_n, x_n\rangle\overline{\langle Sx_n, x_n\rangle}=0$.

(ii)$\Rightarrow$ (i): Let $\{x_n\}\in M^*_{\omega(T+S)}$ and  $ \{y_n\}\in M^*_{\omega (T-S)}$. Hence
$$\omega(T+S)=\lim_n|\langle (T+S) x_n, x_n\rangle|\,\,\text{and}\,\,\omega(T-S)=\lim_n|\langle (T-S) y_n, y_n\rangle|.$$
 By the assumption, we have
$$\omega^2(T+S)\geq\lim_n|\langle (T+S) y_n, y_n\rangle|^2=\lim_n\Big(|\langle Ty_n, y_n\rangle|^2+|\langle Sy_n, y_n\rangle|^2\Big)=\omega^2(T-S)$$
and
$$\omega^2(T-S)\geq\lim_n|\langle (T-S) x_n, x_n\rangle|^2= \lim_n\Big(|\langle Tx_n, x_n\rangle|^2+|\langle Sx_n, x_n\rangle|^2\Big)=\omega^2(T+S)$$
Therefore $\omega(T-S)=\omega(T+S)$, as well as $\{x_n\}\in M^*_{\omega(T-S)}$ and  $ \{y_n\}\in M^*_{\omega (T+S)}$. 
Thus $M^*_{\omega(T-S)}=M^*_{\omega (T+S)}$.
\end{proof}

\begin{corollary}\label{32}
If $\omega(T+S)=\omega(T-S)$ and $T\perp_{\omega p}S$, then the following statements are equivalent:

(i) $T\parallel_{\omega} S$

(ii) $T+S\parallel_{\omega} T-S$
\begin{proof}
(i)$\Rightarrow$ (ii): Let $T\parallel_{\omega} S$. Then by Theorem \ref{31}, there exists a sequence of unit vectors $\{x_n\}$ such that  $\lim_n|\langle Tx_n, x_n\rangle|=\omega(T)$  and  $\lim_n|\langle Sx_n, x_n\rangle|=\omega(S)$. By the assumption $T\perp_{\omega p}S$, we have
$$\omega^2(T)+\omega^2(S)=\omega^2(T+S)\geq\lim_n|\langle (T+S) x_n, x_n\rangle|^2=\omega^2(T)+\lim_n 2{\rm  Re}  \langle Tx_n, x_n\rangle\overline{\langle Sx_n, x_n\rangle}+\omega^2(S).$$
Thus $\lim_n 2Re \langle Tx_n, x_n\rangle\overline{\langle Sx_n, x_n\rangle}\leq0$.
By a similar computation for $\omega^2(T-S)$,  we get
$$\omega^2(T)+\omega^2(S)\geq\omega^2(T-S)\geq\lim_n|\langle (T-S) x_n, x_n\rangle|^2=\omega^2(T)-\lim_n 2{\rm  Re}  \langle Tx_n, x_n\rangle\overline{\langle Sx_n, x_n\rangle}+\omega^2(S).$$
Thus
$- \lim_n 2{\rm  Re}  \langle Tx_n, x_n\rangle\overline{\langle Sx_n, x_n\rangle}\leq 0$. Hence $\lim_n {\rm  Re} \langle Tx_n, x_n\rangle\overline{\langle Sx_n, x_n\rangle}=0$. Therefore $\{x_n\}\in M^*_{\omega(T+S)}\cap M^*_{\omega (T-S)}$, whence $T+S\parallel_{\omega}T-S$.

(ii)$\Rightarrow$(i):
Let $T+S\parallel_{\omega}T-S$, then there exists a sequence $\{x_n\}\in M^*_{\omega(T+S)}\cap M^*_{\omega (T-S)}$ such that
\begin{align*}
\omega^2(T)+\omega^2(S)=\omega^2(T+S)&=\lim_n|\langle (T+S) x_n, x_n\rangle|^2\\
&=\lim_n \Big(|\langle Tx_n, x_n\rangle|^2+2{\rm  Re} \langle Tx_n, x_n\rangle\overline{\langle Sx_n, x_n\rangle}+|\langle Sx_n, x_n\rangle|^2\Big)\\
&\leq \omega^2(T)+\lim_n2Re \langle Tx_n, x_n\rangle\overline{\langle Sx_n, x_n\rangle}+\omega^2(S),
\end{align*}
similary
$$\omega^2(T)+\omega^2(S)=\omega^2(T+S)=\omega^2(T-S)\leq \omega^2(T)-\lim_n2Re \langle Tx_n, x_n\rangle\overline{\langle Sx_n, x_n\rangle}+\omega^2(S).$$
Therefore $\lim_n Re \langle Tx_n, x_n\rangle\overline{\langle Sx_n, x_n\rangle}=0$ and
$\displaystyle\lim_n|\langle Tx_n, x_n\rangle|=\omega(T),\,\, \lim_n|\langle Sx_n, x_n\rangle|=\omega(S).$
Hence $T\parallel_{\omega} S$.
\end{proof}
\end{corollary}
\begin{remark}
If $S\parallel_{\omega} T$ and $T+S \parallel _{\omega} T-S$, then by an argument as in Corollary \ref{32}, we can show that
$T\perp_{\omega p} S$ if and only if $\omega(T+S)=\omega(T-S)$.
\end{remark}

\end{document}